\documentclass[letterpaper,10pt,twoside]{article}

\usepackage{graphicx}
\usepackage{amsmath,amssymb,amsthm, verbatim}
\usepackage{enumerate}
\usepackage{mathrsfs}

\newtheorem{prop}{Proposition}
\newtheorem{thm}[prop]{Theorem}

\newtheorem{lem}[prop]{Lemma}

\newcommand{\two}{2}
\newcommand{\defeq}{:=}
\newcommand{\Q}{\mathbb{Q}}
\newcommand{\card}{\mbox{card}}
\newcommand{\aut}{\mbox{Aut}}
\newcommand{\autv}{\mbox{Aut}_{\text{vertex}}}
\newcommand{\aute}{\mbox{Aut}_{\text{edge}}}
\newcommand{\ii}{\mbox{id}}
\newcommand{\graph}{G}
\newcommand{\graphH}{H}
\newcommand{\legs}{e}
\newcommand{\Vb}{W_{{\tiny\mbox{biconn}}}}
\newcommand{\Ve}{W_{{\tiny\mbox{$\two$-edge}}}}
\newcommand{\Vc}{W_{{\tiny\mbox{conn}}}}
\newcommand{\Vd}{W_{{\tiny\mbox{disconn}}}}
\newcommand{\multiedges}{\rho}
\newcommand{\lp}{k}
\newcommand{\vertex}{n}
\newcommand{\edge}{m}
\newcommand{\ext}{s}
\newcommand{\Qi}{q_{i}^{(\rho)}}

\newcommand{\Qti}{{q_i}^{(\rho)}_{\ge 1}}
\newcommand{\Qtifac}{{\hat{q_i}}^{(\rho)}_{\ge 1}}
\newcommand{\Qcut}{\hat{q}^{(\rho)}_{{{\tiny\mbox{cutvertex}}}_{\ge 1}}}
\newcommand{\Eij}{l_{i,j}}
\newcommand{\Ein}{l_{i,n+1}}

\newcommand{\ri}{r_i}
\newcommand{\rifac}{\hat{r_i}}
\newcommand{\rj}{r_{j}}
\newcommand{\rbox}{r_{{\tiny\mbox{cutvertex}}}}
\newcommand{\rboxfac}{\hat{r}_{{\tiny\mbox{cutvertex}}}}

\newcommand{\shift}{\Xi_i}
\newcommand{\shiftj}{\Xi_j}
\newcommand{\si}{s_i}
\newcommand{\sione}{{s_i}_{\ge 1}}
\newcommand{\sifac}{{\hat{s_i}}_{\ge 1}}
\newcommand{\contraction}{c_{\hat{\graph}}}
\newcommand{\Ext}{E_{{\tiny\mbox{ext}}}}
\newcommand{\Int}{E_{{\tiny\mbox{int}}}}
\newcommand{\Exth}{\hat{E}_{{\tiny\mbox{ext}}}}
\newcommand{\Inth}{\hat{E}_{{\tiny\mbox{int}}}}

\newcommand{\legsii}{\mathcal{E}_{{\tiny\mbox{int}},i}}
\newcommand{\legsiifac}{\hat{\mathcal{E}}_{{\tiny\mbox{int}},i}}

\newcommand{\edgesii}{\mathcal{L}_{{\tiny\mbox{int}},i}}

\newcommand{\edgesiia}{\mathcal{L}_{{\tiny\mbox{int}},i,a}}
\newcommand{\edgesiil}{\mathcal{L}_{{\tiny\mbox{int}},l}}
\newcommand{\edgesiiB}{\mathcal{L}_{{\tiny\mbox{int}}}^\mathcal{B}}
\newcommand{\edgesie}{\mathcal{L}_{{\tiny\mbox{ext}},i}}
\newcommand{\edgese}{\mathcal{L}_{{\tiny\mbox{ext}}}}

\newcommand{\partii}{\mathcal{I}_{{\tiny\mbox{int}}}}
\newcommand{\partiifac}{\hat{\mathcal{I}}_{{\tiny\mbox{int}}}}

\newcommand{\partiib}{\mathcal{I}_{\mathcal{B}_i}}
\newcommand{\legsfree}{\mathcal{E}_{{\tiny\mbox{free}}}}
\newcommand{\mapi}{\varphi_{{\tiny\mbox{int}}}}
\newcommand{\mape}{\varphi_{{\tiny\mbox{ext}}}}
\newcommand{\mapeii}{\varphi_{{\tiny\mbox{ext}}}^*}
\newcommand{\mapih}{\hat{\varphi}_{{\tiny\mbox{int}}}}
\newcommand{\mapeh}{\hat{\varphi}_{{\tiny\mbox{ext}}}}

\newcommand{\Dc}{\beta_{{\tiny\mbox{conn}}}}
\newcommand{\Db}{\beta_{{\tiny\mbox{biconn}}}}
\newcommand{\De}{\beta_{{\tiny\mbox{$\two$-edge}}}}
\newcommand{\Dec}{\beta_{{\tiny\mbox{$\two$-edge}},C}}

\newcommand{\Daux}{\beta}

\newcommand{\extb}{\xi}
\newcommand{\extc}{\psi}
\newcommand{\A}{\mathscr{A}}
\newcommand{\B}{\mathscr{B}}
\newcommand{\C}{\mathscr{C}}
\newcommand{\D}{\mathscr{D}}
\newcommand{\CC}{\mathcal{C}}
\newcommand{\BB}{\mathcal{B}}

\newcommand{\chietprime}{[\chi_e^{t'}]^2}
\newcommand{\chies}{[\chi_e^{2s}]^2}
\newcommand{\chieinvtprime}{[{\chi_e^{t'}}^{-1}]^2}
\newcommand{\chivu}{[\chi_v^{\uu}]^2}
\newcommand{\chiet}{[\chi_e^{t}]^2}
\newcommand{\chivi}{[\chi_v^i]^2}

\newcommand{\uu}{n'}
\newcommand{\comp}{\mu}
\newcommand{\vmin}{n_{{\tiny\mbox{min}}}}
\newcommand{\lmin}{k_{{\tiny\mbox{min}}}}

\allowdisplaybreaks

\title{\textbf{On the decomposition of connected graphs into their biconnected components}}
\author{\^Angela Mestre\footnote{Present address: Centro de Estruturas Lineares e Combinat\'orias,
Complexo Interdisciplinar da Universidade de Lisboa,
Av. Prof. Gama Pinto 2, 1649-003 Lisboa Portugal;
email: mestre@alf1.cii.fc.ul.pt.}\\
Institut de Min\'eralogie et de Physique des Milieux Condens\'es,\\
Universit\'e Pierre et Marie Curie,
Campus Boucicaut,\\
140 rue de Lourmel, F-75015 Paris\\ France
}

\begin{document}

\maketitle

\begin{abstract}
We give a recursion formula to generate all equivalence classes of biconnected
graphs with 
coefficients
given by
the inverses of the orders of their groups of automorphisms. 
We
give a linear map to produce a connected graph with say,
$\mu$, biconnected components from one  with $\mu-1$ biconnected components.  We use
such map to extend the aforesaid result to 
connected or $\two$-edge connected graphs. 
The underlying algorithms are amenable to computer implementation.

\end{abstract}

\bigskip
\bigskip

\section{Introduction}\label{se:intro}
As pointed out in \cite{read}, generating graphs  may be useful for numerous reasons. These include giving  more insight into enumerative problems or the  study of some properties of graphs. Problems of graph generation may also suggest conjectures or point out counterexamples.

In particular, the problem of generating graphs taking into account their symmetries was considered as early as  the 19th-century \cite{Jordan}.  Sometimes such problem is so that graphs are weighted by scalars given by the inverses of the orders of their groups of automorphisms. One instance is  \cite{Polya} (page $\two$09).  Many other examples may be found in  mathematical physics (see for instance \cite{ItZu:qft} and references therein).
In this context, the problem is traditionally dealt with via generating functions and functional derivatives. 
However, any method to straightforwardly manipulate graphs may actually be used. 
In particular, the main results of \cite{MeOe:npoint} and \cite{MeOe:loop} are recursion formulas to generate all equivalence classes of trees and connected graphs (with multiple edges and loops allowed), respectively, 
via  Hopf algebra. The key feature is that the sum of
 the coefficients of all graphs in the same equivalence class is given by the inverse of the order of
their automorphism group. 

Furthermore,  in \cite{Me:classes} the algorithm underlying the main result of \cite{MeOe:loop} was translated to the language of graph theory.  To this end,  
basic graph transformations whose action mirrors that of the Hopf algebra structures considered in the latter paper,  were given. 
The result  was  also extended to further classes of connected graphs, namely, $\two$-edge connected, simple  and loopless graphs.

In the following, let \emph{graphs} be loopless graphs with external edges allowed. That is, edges which are connected to vertices only at one end. In the present paper, we generalize formula (5) of \cite{Me:classes} to 
  biconnected graphs.
Moreover, we give
a linear map to produce connected graphs from connected ones by increasing the number of their 
biconnected components 
by one unit. 
We use this map  to give an algorithm to generate all equivalence classes of
connected or $\two$-edge connected graphs 
with the exact coefficients. This  is so that generated graphs are automatically decomposed into their biconnected components. 
The proof proceeds as  suggested in \cite{MeOe:npoint}.
That is, given an arbitrary  equivalence class whose representative is a graph  on $\edge$ internal edges, say, $\graph$, we show that every one of the $\edge$ internal edges of the graph $\graph$
 adds $1/(\edge\cdot \vert\aut(\graph)\vert)$ 
to the sum of the coefficients of all graphs isomorphic to $\graph$. 
To this end, we use the 
fact that vertices carrying (labeled) external edges are  held fixed under any automorphism.

This paper is organized as follows: Section~\ref{sec:basics} reviews the basic concepts of graph theory underlying much of the paper. Section~\ref{sec:linear} contains the definitions of the basic linear maps to be  used in  the following sections. Section \ref{sec:biconn}  presents an algorithm to generate biconnected graphs and gives some examples. Section  \ref{sec:conn} extends the result to connected and $\two$-edge connected graphs.

\section{Basics}
\label{sec:basics}
We briefly review the basic concepts of graph theory that are relevant for the following sections.  More details may be found in any standard textbook on graph theory such as \cite{diestel}, or in \cite{Me:classes} for the treatment of graphs with external edges allowed.
This section overlaps Section $\two$  of 
\cite{Me:classes} except for
the concept of biconnected graph. However, as the present paper only considers loopless graphs,  the definition of graph  given in that paper 
specializes here for loopless graphs.

\bigskip

  Let  $A$ and $B$ denote sets.  By  $[A,B]$, we denote the set of all unordered pairs of elements of $A$ and $B$, $\{\{a,b\}\vert a\in A, b\in B\}$. In particular, by  $[A]^\two\defeq[A,A]$, we denote the set of all $\two$-element subsets of $A$.   Also, by $\two^A$, we denote the power set of $A$, i.e., the set of all subsets of $A$.  By $\card(A)$, we denote the cardinality of the set $A$. Furthermore, we recall  that  the symmetric difference of the sets $A$ and $B$ is given by $A\triangle B\defeq(A\cup B)\backslash (A\cap B)$. Finally, given a graph $\graph$, by $\vert\aut(\graph)\vert$, we denote the order of its automorphism group $\aut(\graph)$.
\medskip

Let  $\mathcal{V}=\{v_i\}_{i\in\mathbb{N}}$ and $\mathcal{K}=\{\legs_a\}_{a\in\mathbb{N}}$ be infinite sets so that $\mathcal{V}\cap\mathcal{K}=\emptyset$. Let $V\subset\mathcal{V}$; $V\neq\emptyset$ and $K\subset\mathcal{K}$ be finite
sets. 
Let  $E=\Int\cup\Ext\subseteq[K]^\two$ and $\Int\cap\Ext=\emptyset$. Also, let the elements of $E$ satisfy
$\{\legs_a,\legs_{a'}\}\cap\{\legs_b,\legs_{b'}\}=\emptyset$. That is, $\legs_a,\legs_{a'}\neq\legs_b,\legs_{b'}$. 
In this context, a  \emph{graph} is a triple $\graph=(V,K,E)$ together with  the following  maps: 
\begin{enumerate}[(a)]
\item$\mapi:\Int\rightarrow [V]^\two; \{\legs_a,\legs_{a'}\}\mapsto \{v_i,v_{i'}\}$; 
\item $\mape:\Ext\rightarrow [V,K]; \{\legs_a,\legs_{a'}\}\mapsto \{v_i,\legs_{a'}\}$. 
\end{enumerate}
The elements of $V$, $E$ and $K$ are called \emph{vertices}, \emph{edges}, and \emph{ends of edges}, respectively. In particular, the elements of $\Int$ and $\Ext$ are called  \emph{internal edges} and \emph{external edges}, respectively. 
 The \emph{degree} of a vertex is the number of ends of edges
assigned to the vertex. Two distinct vertices  connected together by one or more internal edges, are said to be \emph{adjacent}. Two or more internal edges connecting the same pair of distinct vertices together, are called \emph{multiple edges}. Furthermore, let $\card(\Ext)=\ext$. Let $L=\{x_1,\ldots,x_{\ext}\}$ be a \emph{label} set. 
A \emph{labeling} of the external edges of the graph $\graph$, is an injective map $l:\Ext\rightarrow [K,L]; \{\legs_a,\legs_{a'}\}\mapsto \{\legs_a,x_z\}$, where $z\in\{1,\ldots,\ext\}$. A graph
$\graph^*=(V^*,K^*,E^*)$; $E^*=\Int^*\cup\Ext^*$, together with the maps $\mapi^*$ and $\mape^*$ is called a \emph{subgraph} of a graph
$\graph=(V,K,E)$; $E=\Int\cup\Ext$, together with the maps $\mapi$ and $\mape$ if $V^*\subseteq V$, $K^*\subseteq K$,
$E^*\subseteq E$ and $\mapi^*=\mapi\rvert_{\Int^*}$, $\mape^*=\mape\rvert_{\Ext^*}$.

A \emph{path}  is  a graph $P=(V,K,\Int)$; $V=\{v_1,\ldots,v_{\vertex}\}$, $\vertex\defeq\card(V)>1$, together with the map $\mapi$
so that   $\mapi(\Int)=\{\{v_1,v_\two\},\{v_\two, v_3\},\ldots,\{v_{n-1},v_\vertex\}\}$ and the vertices $v_1$ and $v_\vertex$ have degree 1, while the vertices $v_\two,\ldots,v_{\vertex-1}$ have degree $\two$. In this context, the vertices $v_1$ and $v_\vertex$ are called the \emph{end
 point} vertices, while the vertices $v_\two,\ldots,v_{\vertex-1}$ are called the \emph{inner} vertices.
A \emph{cycle}   is a  graph $C=(V',K',\Int')$; $V'=\{v_1,\ldots,v_{\vertex}\}$,  together with the map $\mapi'$
 so that   $\mapi'(\Int')=\{\{v_1,v_\two\},\{v_\two, v_3\},\ldots,\{v_{\vertex-1},v_\vertex\}, \{v_\vertex, v_1\}\}$ and every vertex has degree $\two$. 
A graph is said to be \emph{connected} if every pair of vertices is joined by a path. Otherwise, it is \emph{disconnected}.

Given a graph  $\graph=(V,K,E)$; $E=\Int\cup\Ext$, together with the maps $\mapi$ and $\mape$, a maximal connected subgraph of the graph $\graph$
is called a {\em component}.  Moreover, the set $\two^{\Int}$ is a vector space over the field $\mathbb{Z}_\two$ so that vector addition is given by the  symmetric difference. The \emph{cycle space} $\CC$ of the graph $\graph$ is defined as the subspace of $\two^{\Int}$ generated by all the cycles in $\graph$. The dimension of $\CC$ is  called the \emph{cyclomatic number} of the graph $\graph$. Let $\lp\defeq\dim\CC$, $\vertex\defeq\card(V)$ and  $\edge\defeq\card(\Int)$. Then,  $\lp=\edge-\vertex+c$, where $c$ denotes the number of connected components of the graph $\graph$ \cite{Kirch}.  

Furthermore, 
given a connected graph, a vertex 
whose removal (together with attached edges) disconnects the graph is called a {\em cut vertex}. 
A graph is said to be {\em $\two$-connected} (resp. {\em $\two$-edge connected}) if it remains connected after erasing any vertex   (resp. any internal edge). 
A  $\two$-connected graph (resp. $\two$-edge connected graph) is also called \emph{biconnected} (resp. {\em edge-biconnected}). 
 Furthermore, a \emph{biconnected component} of a  connected graph  is a maximal biconnected subgraph (see \cite{handbook} 6.4 for instance). 

Now, let $L=\{x_1,\ldots,x_{\ext}\}$ be a finite label set. Let $\graph=(V,K,E)$; $E=\Int\cup\Ext$,  $\card(\Ext)=\ext$, together with the maps $\mapi$ and $\mape$, and $\graph^*=(V^*,K^*,E^*)$;  $E^*=\Int^*\cup\Ext^*$,  $\card(\Ext^*)=\ext$, together with the maps $\mapi^*$ and $\mape^*$ denote two graphs. Let $l:\Ext\rightarrow [K,L]$ and $l^*:\Ext^*\rightarrow [K^*,L]$ be labelings of the elements of $\Ext$ and $\Ext^*$, respectively.  An \emph{isomorphism} between the graphs $\graph$ and $\graph^*$ is a bijection $\psi_V:V\to V^*$ and  a bijection $\psi_K:K\to K^*$  which satisfy the following three conditions:
\begin{enumerate}[(a)]
\item $\mapi(\{\legs_{a},\legs_{a'}\})=\{v_i,v_{i'}\}$ iff $\mapi^*(\{\psi_K(\legs_{a}),\psi_K(\legs_{a'})\})=\{\psi_V(v_i),\psi_V(v_{i'})\}$; 
\item $\mape(\{\legs_{a},\legs_{a'}\})=\{v_i,\legs_{a'}\}$ iff $\mape^*(\{\psi_K(\legs_{a}),\psi_K(\legs_{a'})\})=\{\psi_V(v_{i}),\psi_K(\legs_{a'}) \}$; 
\item $L\cap l(\{\legs_a,\legs_{a'}\})=L\cap l^*(\{\psi_K(\legs_a),\psi_K(\legs_{a'})\})$.
\end{enumerate}
An isomorphism  defines an equivalence relation on graphs.
A \emph{vertex} (resp. \emph{edge}) isomorphism between the graphs $\graph$ and $\graph^*$  is an isomorphism
so that $\psi_K$ (resp. $\psi_V$) is the identity map. 
In this context, a \emph{symmetry} of a graph $\graph$ is an isomorphism
 of the graph  onto itself (i.e., an \emph{automorphism}). 
 A \emph{vertex symmetry} (resp. \emph{edge symmetry}) of
a   graph $\graph$ is a vertex (resp. edge) automorphism of the graph.  Given a graph $\graph$, let $\autv(\graph)$ and $\aute(\graph)$
denote the groups of vertex and edge automorphisms, respectively. Then,  
$\vert\aut(\graph)\vert=
\vert\autv(\graph)\vert\cdot
\vert\aute(\graph)\vert$ (a proof is given in \cite{MeOe:loop} for instance).

\section{Elementary linear transformations}\label{sec:linear}
We introduce the elementary linear maps to be used in the following. 

\bigskip

Given an arbitrary set $X$, by $\Q X$, we  denote the free vector space on the  set  $X$ over $\Q$. 
By  $\ii_X:X\to X;x\mapsto x$, we denote the identity map.
Given maps $f:X\to X^*$ and $g:Y\to Y^*$, by $[f,g]$, we denote the map  $[f,g]:[X,Y]\to[X^*,Y^*];\{x,y\}\mapsto\{f(x),g(y)\}$ with $[f]^\two\defeq[f,f]$. 
 
\medskip

Let  $\mathcal{V}=\{v_i\}_{i\in\mathbb{N}}$ and $\mathcal{K}=\{\legs_a\}_{a\in\mathbb{N}}$ be infinite sets so that $\mathcal{V}\cap\mathcal{K}=\emptyset$. Fix an integer 
$\ext\ge 0$. Let $L=\{x_1,\ldots,x_\ext\}$ be a   label set. For all integers $\vertex\ge 1$ and $\lp\ge 0$, by $\Vc^{\vertex,\lp,\ext}$ (resp. $\Vd^{\vertex,\lp,\ext}$), we denote the set  of all
(loopless) connected graphs (resp. disconnected graphs with two components)  with $\vertex$ vertices,  cyclomatic number $\lp$ and $\ext$ external edges whose free ends are labeled $x_1,\ldots,x_\ext$.  In all that follows, let $V=\{v_1,\ldots,v_\vertex\}\subset\mathcal{V}$, $K=\{\legs_1,\ldots,\legs_t\}\subset\mathcal{K}$ and $E=\Int\cup\Ext\subseteq[K]^2$ be the sets of vertices, ends of edges and edges, respectively, of all elements of $\Vc^{\vertex,\lp,\ext}$ (resp. $\Vd^{\vertex,\lp,\ext}$), so that $\card(\Int)=\lp+\vertex-1$ (resp. $\card(\Int)=\lp+\vertex-\two$)  and $\card(\Ext)=\ext$. 
Also,  let $l:\Ext\rightarrow [K,L]$ be a labeling of  their external edges. 
 Finally,  by $\Vb^{\vertex,\lp,\ext}$ and $\Ve^{\vertex,\lp,\ext}$,  we denote the subsets of $\Vc^{\vertex,\lp,\ext}$ whose elements are biconnected and $\two$-edge connected graphs, respectively. 

\medskip

We begin by briefly recalling the linear maps $\extb_{\Ext,V}$, $\Eij$ and $\sione$   introduced in Sections 3 and 5.3.4 of \cite{Me:classes}. We refer the reader to that paper for the precise definitions. 

\begin{enumerate}[(i)]
\item \emph{Distributing external edges between all elements of a given vertex subset in all possible ways}:
Let $\graph=(V,K,E)$;  
$E=\Int\cup\Ext$, 
together with the maps $\mapi$ and $\mape$ denote a graph in $\Vc^{\vertex,\lp,\ext}$.  Let $V'=\{v_{z_1},\ldots,v_{z_{\vertex'}}\}\subseteq V$. 
Let $K'\subset\mathcal{K}$ be a finite set so that 
$K\cap K'=\emptyset$.  Also, let $\Ext'\subseteq [K']^\two$; $\ext'\defeq\card(\Ext')$. Assume that the elements of $\Ext'$  satisfy  $\{\legs_a,\legs_{a'}\}\cap\{\legs_{b},\legs_{b'}\}=\emptyset$.  Let $L'=\{x_{\ext+1},\ldots,x_{\ext+\ext'}\}$ be a  label set so that $L\cap L'=\emptyset$. Let $l':\Ext'\rightarrow [K',L']$ be a labeling of the elements of $\Ext'$. Finally, let  $\mathcal{I}^{\vertex'}_{\Ext'}$ denote the set of all ordered partitions of the set $\Ext'$  into $\vertex'$ disjoint subsets:  $\mathcal{I}^{\vertex'}_{\Ext'}=\{(\Ext'^{(1)},\ldots,\Ext'^{(\vertex')})\vert\Ext'^{(1)}\cup\ldots\cup \Ext'^{(\vertex')}=\Ext' \quad \mbox{and}\quad \Ext'^{{(i)}}\cap \Ext'^{(j)}=\emptyset\,,\forall i, j\in\{1,\ldots,\vertex'\}\quad\mbox{with}\quad i\neq j\}$.
  In this context, the maps
$$\extb_{\Ext',V'} :\Q \Vc^{\vertex,\lp,\ext}\rightarrow\Q \Vc^{\vertex,\lp,\ext+\ext'};\graph\mapsto\sum_{(\Ext'^{(1)},\ldots,\Ext'^{(\vertex')})\in \mathcal{I}^{\vertex'}_{\Ext'}}\graph_{(\Ext'^{(1)},\ldots,\Ext'^{(\vertex')})}$$ are defined to produce each of 
 the graphs $\graph_{(\Ext'^{(1)},\ldots,\Ext'^{(\vertex')})}$ from the graph $\graph$ by assigning all elements of  $\Ext'^{(j)}$ to the vertex $v_{z_j}$ for all $j\in\{1,\ldots,\vertex'\}$.

\item The maps $\Eij:\Q \Vc^{\vertex,\lp,\ext}\rightarrow\Q \Vc^{\vertex,\lp+1,\ext};\graph\mapsto\graph^*$ are defined to
 produce
the graph $\graph^*$  from the graph $\graph$ by connecting (or reconnecting) 
the vertices $v_i$ and $v_j$ with an internal edge for all   $i,j\in\{1,\ldots,\vertex\}$ with $i\neq j$.

\item
\begin{enumerate}
\item
We define the maps $\sione$ from the maps $\si$ given in Section 3 of \cite{Me:classes}, by restricting the image of the latter to graphs without isolated vertices. More precisely,
let $\graph=(V,K,E)$; 
$E=\Int\cup\Ext$, 
together with the maps $\mapi$ and $\mape$ denote a graph in $\Vc^{\vertex,\lp,\ext}$. Let $\legsii\subset K$ be the set of ends of internal edges assigned to the vertex $v_i\in V$; $i\in\{1,\ldots,\vertex\}$. Let $\partii^\two$ denote the set of all ordered partitions of the set $\legsii$  into two non-empty disjoint sets: $\partii^\two=\{(\legsii^{(1)},\legsii^{(\two)})\vert \legsii^{(1)}, \legsii^{(\two)}\neq\emptyset\,,\quad\legsii^{(1)}\cup \legsii^{(\two)}=\legsii\quad  \mbox{and}\quad \legsii^{{(1)}}\cap \legsii^{(\two)}=\emptyset\}$. Moreover, let $\edgesie\subset\Ext$ denote the set of external edges assigned to the vertex $v_i$.  In this context, for all $i\in\{1,\ldots,\vertex\}$, define the maps
\begin{eqnarray*}\lefteqn{\sione:\Q \Vc^{\vertex,\lp,\ext}\rightarrow\Q \Vc^{\vertex+1,\lp-1,\ext}\cup\Q \Vd^{\vertex+1,\lp,\ext};}\\
&& \graph\mapsto\left\{\begin{array}{ccc}0 \quad \mbox{if} \quad 0\leq\card(\legsii)<\two\,;\\
\extb_{\edgesie,\{v_i,v_{\vertex+1}\}}\biggl(\sum_{(\legsii^{(1)},\legsii^{(\two)})\in\partii^{\two}}\graph_{(\legsii^{(1)},\legsii^{(\two)})}\biggr)\,\, \mbox{otherwise}\,, &\end{array}\right.\end{eqnarray*}
where each of the graphs $\graph_{(\legsii^{(1)},\legsii^{(\two)})}$ is produced  from the graph $\graph$ as follows: (a) split the vertex $v_i$ into two vertices, namely, $v_i$ and $v_{\vertex+1}$; (b) assign the ends of edges in $\legsii^{(1)}$ and $\legsii^{(\two)}$ to $v_i$ and $v_{\vertex+1}$, respectively.
\item 
Let 
$\mathcal{B}_i=\{\graph_i^1,\ldots,\graph_i^{\vertex_i}\}$ with $\vertex_i\ge 1$, be the set of biconnected components of the graph $\graph$ containing the vertex $v_i$. That is,  $v_i\in V_i^j$,   where $V_i^j$ denotes the vertex set of the graph $\graph_i^j$ for all $j\in\{1,\ldots, \vertex_i\}$.  Let $K^j$ denote the set of ends of edges of the graph $\graph_i^j$.  Let $\legsii^j\subset K^j$ be the set of ends of edges assigned to the vertex $v_i$ which belong to the graph $\graph_i^j$. Hence, $\legsii=\cup_{j=1}^{\vertex_i}\legsii^j$ and $\legsii^j\cap\legsii^m=\emptyset$ for all $j,m=1,\ldots, \vertex_i$ with $j\neq m$. In this context, we define the maps $\sifac$ by restricting the image of $\sione$ to graphs obtained  by replacing in the above definition the set $\partii^{\two}$ by  
$\partiifac^\two=\{(\legsiifac^{(1)},\legsiifac^{(\two)})\vert \legsiifac^{(1)}, \legsiifac^{(\two)}\neq\emptyset\,,\quad\legsiifac^{(1)}\cup \legsiifac^{(\two)}=\legsii\,, \legsiifac^{{(1)}}\cap \legsiifac^{(\two)}=\emptyset\  \quad  \mbox{and}\quad \legsiifac^{(1)}\cap\legsii^j\neq\emptyset,\legsiifac^{(\two)}\cap\legsii^j\neq\emptyset \forall j=1,\ldots,\vertex_i\} $ with $\sifac(\graph)\defeq 0$ if there exists $ j$ so that $
\card(\legsii^j)<\two$. 

\item
We define the  maps $\Qti$ and  $\Qtifac$ in analogy with the maps $\Qi\defeq\frac{1}{\two}{\Ein^{\multiedges}}\circ\si$ given in \cite{Me:classes} (see also \cite{103} for $\multiedges=1$):
\begin{eqnarray}\label{eq:Qt}
\Qti\defeq\frac{1}{\two(\multiedges-1)!}{\Ein^{\multiedges}}\circ\sione:\Q \Vc^{\vertex,\lp,\ext}\to\mathbb{Q}\Vc^{\vertex+1,\lp+\multiedges-1,\ext}\,,\\
\label{eq:Qtfac}
\Qtifac\defeq\frac{1}{\two(\rho-1)!}{\Ein^{\rho}}\circ\sifac:\Q \Vc^{\vertex,\lp,\ext}\to\mathbb{Q}\Vc^{\vertex+1,\lp+\rho-1,\ext}\,,
\end{eqnarray}
where $\Ein^{\multiedges}$ denotes the $\multiedges$th iterate of $\Ein$ with $\Ein^{0}=\ii$.
\end{enumerate}
We now introduce the following auxiliary map: 
\item 
Fix  integers  $\vertex'>0$ and $1\le i\leq\vertex'$. 
Let $\graph=(V,K,\Int)$
together with the map $\mapi$ 
denote a graph in $\Vc^{\vertex,\lp,0}$. We define the map
$\shift:\graph\mapsto\graph^*$, where the graph $\graph^*$ satisfies the following conditions:
\begin{enumerate}[(a)]
\item $V^*=\chi_v^i(V)$, where $\chi_v^i:v_l\mapsto\left\{\begin{array}{cc}v_i &\mbox{if}\quad  l=1\\ v_{\vertex'+l-1} &\mbox{if}\quad  l\in\{\two,\ldots,\vertex\}\\
\end{array}\right.$ is a bijection;
\item 
$K^*=\chi_e^t(K)$, where $\chi_e^t:\legs_b\mapsto\legs_{t+b}$ is a bijection; 
\item 
$\Int^*=\chiet(\Int)$; 
\item 
$\mapi^*\circ\chiet=\chivi\circ\mapi$.  
\end{enumerate}
The maps $\shift$
are extended to the whole of $\mathbb{Q} \Vc^{\vertex,\lp,0}$ by linearity.
\item
\begin{enumerate}
\item
\emph{Distributing the biconnected components of a
connected graph sharing a vertex, 
between all the vertices of a given biconnected graph in all possible ways}: 
Let $\graph=(V,K,E)$;   
$E=\Int\cup\Ext$, 
together with the maps $\mapi$ and $\mape$ denote a graph in $\Vc^{\vertex,\lp,\ext}$. 
Let $\edgesii\subseteq\Int$ and $\edgesie\subseteq\Ext$ be the sets of internal and external edges, respectively, assigned to the vertex  $v_i\in V$ with $i\in\{1,\ldots,\vertex\}$.
Let $\mathcal{B}_i=\{\graph_i^1,\ldots,\graph_i^{\vertex_i}\}$ be the set of biconnected components of the graph $\graph$ sharing the vertex $v_i$. That is,  $v_i\in V_i^j$,   where $V_i^j$ denotes the vertex set of the graph $\graph_i^j$ for all $j\in\{1,\ldots, \vertex_i\}$.
Let $\partiib^{\uu}$ denote the set of all ordered partitions of the set $\mathcal{B}_i$  into $\uu$ disjoint sets: $\partiib^{\uu}=\{(\mathcal{B}_i^{(1)},\ldots,\mathcal{B}_i^{(\uu)})\vert\mathcal{B}_i^{(1)}\cup\ldots\cup \mathcal{B}_i^{(\uu)}=\mathcal{B}_i \quad \mbox{and}\quad \mathcal{B}_i^{{(l)}}\cap \mathcal{B}_i^{(l')}=\emptyset\,,\forall l, l'\in\{1,\ldots,\uu\}\quad\mbox{with}\quad l\neq l'\}$.
For all $l\in\{1,\ldots,\uu\}$, let $\mathcal{B}_i^{(l)}=\{\graph_{i,1}^{(l)},\ldots,\graph_{i,\vertex_{(l)}}^{(l)}\}$.
Also, let $V^{(l)}_{i,a}\subset V$ denote the vertex set of the graph $\graph_{i,a}^{(l)}$ for all $a\in\{1,\ldots,\vertex_{(l)}\}$. 
Let $\edgesiia^{(l)}\subseteq\edgesii$ and $V'^{(l)}_a\subseteq V^{(l)}_{i,a}\backslash\{v_i\}$ be so that $\mapi(\edgesiia^{(l)})=[v_i,V'^{(l)}_a]$.
Let  $\edgesii^{(l)}\defeq\bigcup_{a=1}^{\vertex_{(l)}}\edgesiia^{(l)}$ and $V'^{(l)}\defeq\bigcup_{a=1}^{\vertex_{(l)}}V'^{(l)}_a$. Finally, let $\graph'$ 
denote a graph in $\Vb^{\uu,\lp',0}$. 
Also, let $\shift(\graph')=(\hat{V},\hat{K},\Inth)$ together with the map  $\mapih$ be the graph obtained from $\graph'$ by applying the map $\shift$ given above. 
In this context, for all $i\in\{1,\ldots,\vertex\}$ define 
\begin{eqnarray*}\lefteqn{\ri^{\graph'}:\Q \Vc^{\vertex,\lp,\ext}\rightarrow\Q \Vc^{\vertex+\uu-1,\lp+\lp',\ext};}\\
&& \graph\mapsto\extb_{\edgesie,\hat{V}}\biggl(\sum_{(\mathcal{B}_i^{(1)},\ldots,\mathcal{B}_i^{(\uu)})\in\partiib^{\uu}}\graph_{(\mathcal{B}_i^{(1)},\ldots, \mathcal{B}_i^{(\uu)})}\biggr)\,,\end{eqnarray*}
where the graphs $\graph_{(\mathcal{B}_i^{(1)},\ldots, \mathcal{B}_i^{(\uu)})}=(V^*,K^*,E^*)$; $E^*=\Int^*\cup\Ext^*$, together with the maps $\mapi^*$ and $\mape^*$ satisfy the following conditions:
\begin{enumerate}[(a)]
\item $V^*=V\backslash\{v_i\}\cup\hat{V}$; 
\item $K^*=K\cup\hat{K}$;
\item $E^*=\Int^*\cup\Ext^*$, where $\Int^*=\Int\cup\Inth$, $\Ext^*=\Ext\backslash\edgesie$; 
\item  $\mapi^*|_{\Int\backslash\edgesii}=\mapi|_{\Int\backslash\edgesii}$, $\mapi^*|_{\Inth}=\mapih|_{\Inth}$;\\
$\mapi^*(\edgesii^{(1)})=[v_i,V'^{(1)}]$ and $\mapi^*(\edgesii^{(l)})=[v_{l+\uu-1},V'^{(l)}]$ for all $l\in\{\two,\ldots,\uu\}$;
\item $\mapeii|_{\Ext\backslash\edgesie}=\mape|_{\Ext\backslash\edgesie}$; 
\item $l^*= l|_{\Ext\backslash\edgesie}$  
is a labeling of the elements of $\Ext^*$. 
\end{enumerate} 
The maps $\ri^{\graph'}$ are extended to the whole of $\Q \Vc^{\vertex,\lp,\ext}$ by linearity. 
For instance, let $C_4$ denote a cycle on four vertices. Figure \ref{fig:rbox} shows the result of applying the map $\rbox^{C_4}$ to the cut vertex of a $\two$-edge connected graph with two biconnected components. 
\begin{figure}[h]
\begin{center}
\includegraphics[width=11cm]{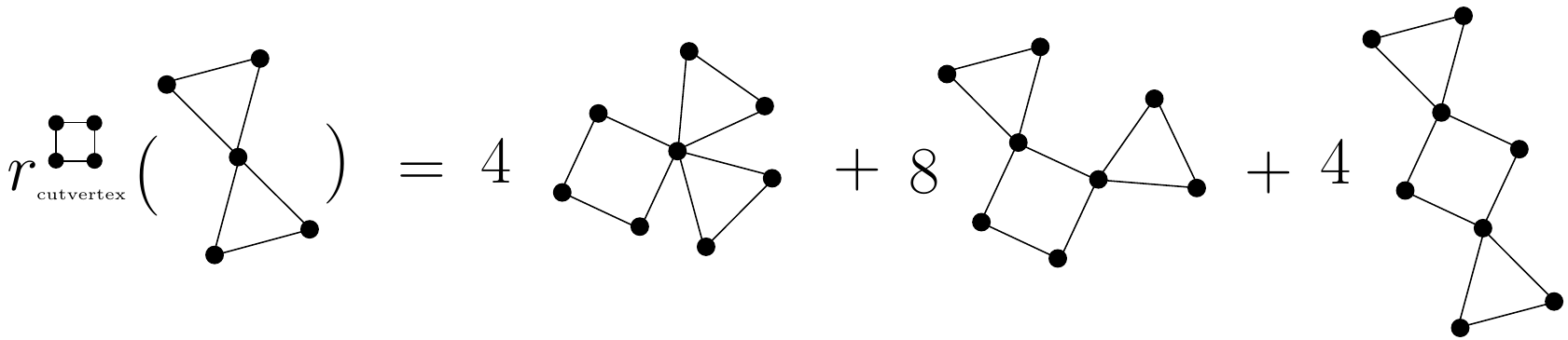}
\end{center}
\caption{Linear combination of graphs obtained by applying the maps $\rbox^{C_4}$ to the cut vertex of the graph consisting of two triangles sharing one vertex.} 
\label{fig:rbox}
\end{figure}
\item
We define the maps $\rifac^\graph$ by restricting the image of $\ri^\graph$ to graphs obtained  by replacing in the definition of the latter the set $\partiib^{\uu}$ by the following set of $\uu$-tuples:  $\{(\mathcal{B}_i,\emptyset,\ldots,\emptyset), (\emptyset, \mathcal{B}_i, \ldots, \emptyset),\ldots, (\emptyset,\ldots,\emptyset,\mathcal{B}_i)\}$.  
\item
Given a linear combination of graphs  $\vartheta=\sum_{\graph\in\Vb^{\vertex,\lp,\ext}} \alpha_\graph \graph$; $\alpha_\graph\in\Q$ in $\Q \Vb^{\vertex,\lp,\ext}$, define 
\begin{equation}\label{eq:Qgen}
\ri^{\vartheta}\defeq\sum_{\graph\in\Vb^{\vertex,\lp,\ext}} \alpha_\graph\, \ri^{\graph}.
\end{equation}
\begin{equation}\label{eq:Qgenfac}
\rifac^{\vartheta}\defeq\sum_{\graph\in\Vb^{\vertex,\lp,\ext}} \alpha_\graph\, \rifac^{\graph}.
\end{equation}
\end{enumerate}
We  proceed to generalize the edge contraction operation  
to  the operation of contracting a biconnected component of a 
connected graph. 
\item 
\emph{Contracting a biconnected component of a connected graph}:
Let $\graph=(V,K,E)$; 
$E=\Int\cup\Ext$, 
together with the maps $\mapi$ and $\mape$ denote a graph in $\Vc^{\vertex,\lp,\ext}$, where $\vertex>1$. 
Consider the following biconnected component of the graph $\graph$: $\hat{\graph}=(\hat{V},\hat{K},\hat{E})$; $\hat{V}=\{v_{i_1},\ldots,v_{i_{\uu}}\}\subseteq V$, $\hat{K}=\{\legs_{j_1},\ldots,\legs_{j_{t'}}\}\subseteq K$; $\hat{E}=\Inth\cup\Exth\subseteq E$, 
together with the maps $\mapih$ and $\mapeh$.    Let $i_1<\ldots<i_{\uu}$ and $j_1<\ldots <j_{t'}$. Also, let $\lp'$ denote the cyclomatic number of the graph $\hat{\graph}$.
Let $\mathcal{B}=\{\graph_l,\ldots,\graph_{\vertex_\BB}\}$ 
be the set of biconnected components of the graph $\graph$ so that   $\hat{V}\cap V_l\neq\emptyset$,  
where $V_l\subset V$ denotes the vertex set of the graph $\graph_l$ for all $l\in\{1,\ldots,\vertex_\BB\}$.
Now, let $\edgesiil\subseteq\Int$ and $V'_l\subseteq V_l\backslash(\hat{V}\cap V_l)$ be so that $\mapi(\edgesiil)\defeq[\hat{V}\cap V_l,V'_l]$; 
$l\in\{1,\ldots,\vertex_\BB\}$. Also, let $\edgesiiB\defeq\bigcup_{l=1}^{\vertex_\BB}\edgesiil$ and $V'^{\mathcal{B}}\defeq\bigcup_{l=1}^{\vertex_\BB}V'_l$.
Finally, let $\legsfree\subset K$ denote the set of free ends of the external edges in $\Exth$.
In this context, define 
$$\contraction:\Q \Vc^{\vertex,\lp,\ext}\rightarrow\Q \Vc^{\vertex-\uu+1,\lp-\lp',\ext};\graph\mapsto\graph^*\,,$$
where the graph $\graph^*=(V^*,K^*,E^*)$; $E^*=\Int^*\cup\Ext^*$, together with the maps $\mapi^*$ and $\mape^*$ satisfies the following conditions:
\begin{enumerate}[(a)]
\item $V^*=\chi_v^{\uu}(V\backslash(\hat{V}\backslash\{v_{i_1}\}))$, where $$\chi_v^{\uu}:v_l\mapsto\left\{\begin{array}{cc}v_l &\mbox{if}\quad  l\in\{1,\ldots,i_\two-1\}\\ v_{l-j+1} &\mbox{if}\quad  l\in\{i_{j}+1,\ldots,i_{j+1}-1\}, j\in\{\two,\ldots,\uu\}\\
v_{l-\uu+1} &\mbox{if}\quad  l\in\{i_{\uu}+1,\ldots,\vertex\}\\
\end{array}\right.$$ is a bijection; 
\item$K^*=\chi_e^{t'}(K\backslash\hat{K})$, where $$\chi_e^{t'}:\legs_b\mapsto\left\{\begin{array}{cc} \legs_{b} &\mbox{if}\quad  b\in\{1,\ldots,j_1-1\}\\\legs_{b-l} &\mbox{if}\quad b\in\{j_l+1,\ldots,j_{l+1}-1\}, l\in\{\two,\ldots,t'-1\}\\\legs_{b-t'} &\mbox{if}\quad b\in\{j_{t'}+1,\ldots,t\}\\\end{array}\right.$$ is a bijection; 
\item$E^*=\Int^*\cup\Ext^*$, where $\Int^*=\chietprime(\Int\backslash\Inth)$, $\Ext^*=\chietprime(\Ext)$; 
\item$\mapi^*\circ\chietprime\vert_{\Int\backslash(\Inth\cup\edgesiiB)}=\chivu\circ\mapi\vert_{\Int\backslash(\Inth\cup\edgesiiB)}$;\\
$\mapi^*(\chietprime(\edgesiiB)=[v_{i_1},\chi_v^{\uu}(V'^{\mathcal{B}})]$;
\item 
$\mape^*\circ\chietprime\vert_{\Ext\backslash\Exth}=[\chi_v^{\uu},\chi_e^{t'}]\circ\mape\vert_{\Ext\backslash\Exth}$;\\
$\mape^*(\chietprime(\Exth))=[v_{i_1},\chi_e^{t'}(\legsfree)]$;
\item 
$l^*=[\chi_e^{t'},\ii_L]\circ l\circ\chieinvtprime:\Ext^*\to[K^*,L]$ 
is a labeling of the elements of $\Ext^*$.
\end{enumerate} 
\item 
 \emph{Erasing all external edges of a connected graph}:
Let $\graph=(V,K,E)$;  
$E=\Int\cup\Ext$, 
together with the maps $\mapi$ and $\mape$ denote a graph in $\Vc^{\vertex,\lp,\ext}$. Let $K_1,K_\two\subset K$ satisfy $K_1\cap K_\two=\emptyset$ and $\Ext=[K_1,K_\two]$. Let $K_1\cup K_\two=\{e_{i_1},\ldots,e_{i_{\two\ext}}\}$  with $i_1<\ldots<i_{\two\ext}$. In this context,   
define 
$$\extc :\Q \Vc^{\vertex,\lp,\ext}\rightarrow\Q \Vc^{\vertex,\lp,0};\graph\mapsto\graph^*\,,$$ 
where the graph $\graph^*=(V^*,K^*,\Int^*)$ together with the map $\mapi^*$ satisfies the following conditions:
\begin{enumerate}[(a)]
\item $V^*=V$;
\item $K^*=\chi_e^{\two\ext}(K\backslash (K_1\cup K_\two))$, where the bijection $\chi_e^{\two\ext}$ is given in (vi); 
\item $E^*=\chies(\Int)$;  
\item $\mapi^*(\chies(\Int))=\mapi$.
\end{enumerate} 
\end{enumerate}

\smallskip
The following lemmas are  now established.

\begin{lem}\label{lem:biconntobiconn}
Fix integers $\ext\ge0$, $\lp>0$ and $\vertex>1$.  Then, for all $i\in\{1,\ldots,\vertex\}$,
$\Qti(\Q\Vb^{\vertex,\lp,\ext})\subseteq\Q\Vb^{\vertex+1,\lp+\multiedges-1,\ext}$.
\end{lem}
\begin{proof}
Let  
$\graph$ denote a graph in $\Vb^{\vertex, \lp,\ext}$.  
Apply the map $\Qti\defeq\frac{1}{\two(\multiedges-1)!}\Ein^\multiedges\circ\sione$ to the graph $\graph$. 
  $\sione(\graph)$ is a linear combination of connected graphs. Therefore, by Lemma 5 of \cite{Me:classes} the graphs in $\Qti(\graph)$ are $\two$-edge connected. Clearly, they must be biconnected in particular. 
\end{proof}
\begin{lem}\label{lem:edgeconntobiconn}
Fix integers $\ext\ge0$, $\lp,>0$ and $\vertex>1$. (a) Let $\graph$ denote a graph in $\Ve^{\vertex,\lp,\ext}$ with only one cut vertex. Then, $\Qcut(\graph)\in\Vb^{\vertex+1,\lp+\multiedges-1,\ext}$. (b) Let $\graph$ denote a graph in $\Ve^{\vertex,\lp,\ext}$ with at least two cut vertices. Then, $\Qti(\graph)-\Qtifac(\graph)\in\Ve^{\vertex+1,\lp+\multiedges-1,\ext}\backslash \Vb^{\vertex+1,\lp+\multiedges-1,\ext}$.
\end{lem}
\begin{proof}
The proof  is trivial.
\end{proof}
\begin{lem}\label{lem:edgeconntoedgeconn}
Fix integers $\ext\ge0$, $\lp,\lp'>0$ and $\vertex,\vertex'>1$. Let $\graph$ denote a graph in $\Vb^{\vertex',\lp',0}$. Then, for all $i\in\{1,\ldots,\vertex\}$, $\ri^{\graph}(\Q\Ve^{\vertex,\lp,\ext})\subseteq\Q\Ve^{\vertex+\vertex'-1,\lp+\lp',\ext}$. 
\end{lem}
\begin{proof}
The proof  is trivial.
\end{proof}

\section{Recursion formulas}
\label{sec:rec}
We give a recursion formula to generate all equivalence classes of biconnected graphs.  In a recursion step, the formula yields a linear combination of  biconnected graphs with the same vertex and cyclomatic numbers. The key feature is that  
the sum of the (rational)
coefficients of all graphs in the same equivalence class corresponds to  the inverse of the order of
their  group of automorphisms. We extend the result to connected or $\two$-edge connected graphs. The underlying algorithms are amenable  to direct implementation via coalgebra in the sense of \cite{Me:coalgebra}.
\subsection{Biconnected graphs}
\label{sec:biconn}
We pick out the terms that generate biconnected graphs in  formula (5) of \cite{Me:classes}.

\bigskip

\begin{thm}\label{thm:biconn}
Fix an integer $\ext\ge0$.    
For all integers $\lp\ge0$ and $\vertex>1$,   define $\Db^{\vertex,\lp,\ext} \in \Q\Vb^{\vertex,\lp,\ext}$
by the following recursion relation:
\begin{itemize}
\item
$\Db^{\two,\lp,\ext}\defeq\frac{1}{\two(\lp+1)!}\extb_{\Ext,\{v_1,v_\two\}}(l_{1,\two}^{\lp}(P_\two))$, where $P_{\two}\in\Vb^{\two,0,0}$ denotes a path on two vertices;
\item $\Db^{\vertex,0,\ext}\defeq 0$, $\vertex> \two$; 
\item
\begin{eqnarray}\nonumber
\lefteqn{\Db^{\vertex,\lp,\ext}  \defeq
\frac{1}{\lp+\vertex-1}\biggl(\sum_{\multiedges=1}^{\lp+1}
\sum_{i=1}^{\vertex-1}\Qti(\Db^{\vertex-1,\lp+1-\multiedges,\ext})+}\\
&&\sum_{j=2}^{\vertex-2}\sum_{\multiedges=1}^{\lp-j+1}\Qcut(\Daux_j^{\vertex-1,\lp+1-\multiedges,\ext})\biggr)\,, 
 \label{eq:recbiconn}
\end{eqnarray}
\end{itemize}
where for all integers $j>1$, $\vertex\ge j+1$ and $\lp \ge j$, $\Daux_j^{\vertex,\lp,\ext}$ is given by the following recursion relation:
\begin{itemize}
\item \begin{equation}\label{eq:aux2}
\Daux_2^{\vertex,\lp,\ext}\defeq\frac{1}{\lp+\vertex-1}\sum_{\lp'=1}^{\lp-1}\sum_{\vertex'=\two}^{\vertex-1}\sum_{i=1}^{\vertex-\vertex'+1}\biggl((\lp'+\vertex'-1)\ri^{\Db^{\vertex',\lp',0}}(\Db^{\vertex-\vertex'+1,\lp-\lp',\ext})\biggr)\,;
\end{equation}

\item \begin{equation}\label{eq:auxj}
\Daux_j^{\vertex,\lp,\ext}\defeq\frac{1}{\lp+\vertex-1}\sum_{\lp'=1}^{\lp-1}\sum_{\vertex'=\two}^{\vertex-1}\biggl((\lp'+\vertex'-1)\rboxfac^{\Db^{\vertex',\lp',0}}(\Daux_{j-1}^{\vertex-\vertex'+1,\lp-\lp',\ext})\biggr)\,.
\end{equation}
\end{itemize}
Then, for fixed values of  $\vertex$ and $\lp$, $\Db^{\vertex,\lp,\ext}=\sum_{\graph\in \Vb^{\vertex,\lp,\ext}}\alpha_\graph\, \graph$;  $\alpha_\graph\in\Q$  for all $\graph\in\Vb^{\vertex,\lp,\ext}$. Moreover, given an arbitrary equivalence class $\C\subseteq\Vb^{\vertex,\lp,\ext}$, the following holds: (i) There exists $\graph\in\C$ so that $\alpha_\graph>0$; (ii) $\sum_{\graph\in\C}\alpha_\graph=1/\vert\aut(\C)\vert$. 
\end{thm}
In formula (\ref{eq:recbiconn}), the 
$\Qcut$ summand does not appear when 
$\vertex<4$ or $\lp<\two$.  

\begin{proof}
Note that $\Daux_j^{\vertex,\lp,\ext}$ is the linear combination of all equivalence classes of $\two$-edge connected graphs with only one cut vertex, $j$ biconnected components, $\vertex$ vertices, cyclomatic number $\lp$ and $\ext$ external edges,  with coefficient given by the inverse of the order of their automorphism group. This is a particular case of Theorem \ref{thm:conn} to be given later on. We refer the reader to the next section for the proof. Thus, by Lemmas \ref{lem:biconntobiconn} and \ref{lem:edgeconntobiconn}, the proof of Theorem \ref{thm:biconn} follows straightforwardly from that of Theorem  18 of \cite{Me:classes}.
\end{proof}

\subsubsection{Examples}\label{sec:appbiconn}
We show the result of computing all mutually non-isomorphic biconnected graphs without external edges as  contributions to $\Db^{\vertex,\lp,0}$ via formula  (\ref{eq:recbiconn}) up to order $\two\le\vertex+\lp\le6$. The coefficients in front  of graphs are the
inverses of the orders of their groups of automorphisms. 
\vspace{1cm}\\
\includegraphics{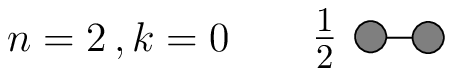}
\vspace{1cm}\\
\includegraphics{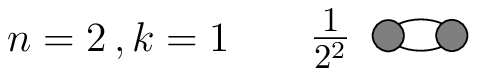}
\vspace{1cm}\\
\includegraphics{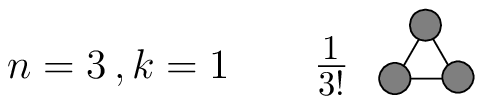}
\vspace{1cm}\\
\includegraphics{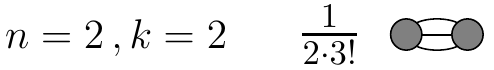}
\vspace{1cm}\\
\includegraphics{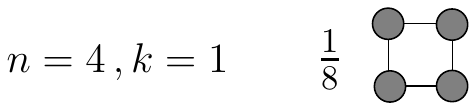}
\vspace{1cm}\\
\includegraphics{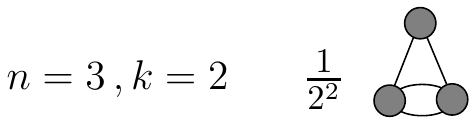}
\vspace{1cm}\\
\includegraphics{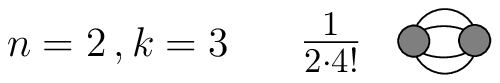}
\vspace{1cm}\\
\includegraphics{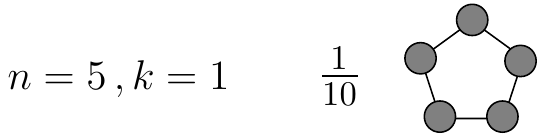}
\vspace{1cm}\\
\includegraphics{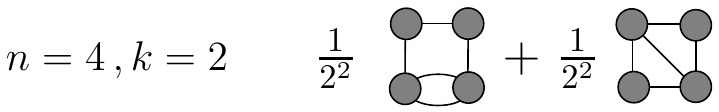}
\vspace{1cm}\\
\includegraphics{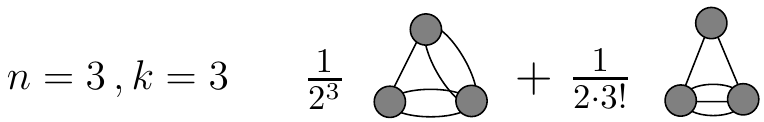}
\vspace{1cm}\\
\includegraphics{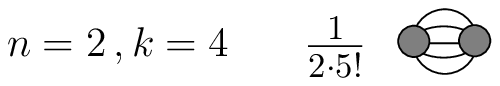}
\vspace{1cm}\\

\subsection{Connected graphs}
\label{sec:conn}
We use  the maps $\ri^{\Db^{\vertex,\lp,\ext}}$ and the linear combination of graphs $\Db^{\vertex,\lp,\ext}\in \Q\Vb^{\vertex,\lp,\ext}$ given by formulas (\ref{eq:Qgen}) and (\ref{eq:recbiconn}), respectively, to generate all equivalence classes of connected graphs. The  underlying algorithm is so that generated graphs are automatically decomposed into their biconnected components.

\begin{thm}\label{thm:conn}
Fix an integer $\ext\ge0$.   
For all integers $\lp\ge0$ and $\vertex> 1$,
define $\Dc^{\vertex,\lp,\ext} \in \Q\Vc^{\vertex,\lp,\ext}$
by the following recursion relation:
\begin{itemize}
\item
$\Dc^{\two,\lp,\ext}\defeq\Db^{\two,\lp,\ext}$; 
\item
\begin{eqnarray}\nonumber
\lefteqn{\Dc^{\vertex,\lp,\ext}\defeq
\Db^{\vertex,\lp,\ext}+\frac{1}{\lp+\vertex-1}\cdot}\\\label{eq:recconn}
&&\sum_{\lp'=0}^{\lp}\sum_{\vertex'=\two}^{\vertex-1}\sum_{i=1}^{\vertex-\vertex'+1}\biggl((\lp'+\vertex'-1)\ri^{\Db^{\vertex',\lp',0}}(\Dc^{\vertex-\vertex'+1,\lp-\lp',\ext})\biggr), \vertex>\two\,.
 \end{eqnarray}
\end{itemize}
Then, for fixed values of  $\vertex$ and $\lp$, $\Dc^{\vertex,\lp,\ext}=\sum_{\graph\in \Vc^{\vertex,\lp,\ext}}\alpha_\graph\, \graph$;  $\alpha_\graph\in\Q$  for all $\graph\in\Vc^{\vertex,\lp,\ext}$. Moreover, given an arbitrary equivalence class $\C\subseteq\Vc^{\vertex,\lp,\ext}$, the following holds: (i) There exists $\graph\in\C$ so that $\alpha_\graph>0$; (ii) $\sum_{\graph\in\C}\alpha_\graph=1/\vert\aut(\C)\vert$. 
\end{thm}
\begin{proof}
Let $\graph$ denote any  connected graph  with $\edge\ge1$ internal edges. We proceed to show that an arbitrary biconnected component of the graph $\graph$ with $\edge'\ge1$ internal edges   adds $\edge'/(\edge\cdot \vert\aut(\graph)\vert)$ to the sum of the coefficients of all graphs isomorphic to the graph $\graph$. 
 As expected, we thus conclude that  every one of the $\edge$ internal edges of the graph $\graph$ contributes $1/(\edge\cdot \vert\aut(\graph)\vert)$ to that sum. 
\begin{lem}
\label{lem:allconn}
Fix integers $\ext\ge0$, $\lp\ge0$ and  $\vertex>1$.
 Let $\Dc^{\vertex,\lp,\ext}=\sum_{\graph\in \Vc^{\vertex,\lp,\ext}}\alpha_\graph\, \graph \in \Q\Vc^{\vertex,\lp,\ext}$ be defined by formula (\ref{eq:recconn}). Let  $\C\subseteq \Vc^{\vertex,\lp,\ext}$ denote an arbitrary equivalence class. Then, there exists $\graph\in\C$ so that $\alpha_\graph>0$. 
\end{lem}
\begin{proof}
The proof proceeds by induction on the number of biconnected components $\comp$. 
By Theorem \ref{thm:biconn}, the statement holds for all graphs in $\Dc^{\vertex,\lp,\ext}$ with only one biconnected component.
We assume the statement to  hold for graphs in $\Dc^{\vertex,\lp,\ext}$ with 
$\comp-1\ge1$ biconnected components. 
Let $\graph$ 
denote any graph in $\C\subseteq\Vc^{\vertex,\lp,\ext}$ with $\comp$ biconnected components. 
Let $\Db^{\vertex',\lp',0}=\sum_{\graph'\in\Vb^{\vertex',\lp',0}}\eta_{\graph'}\graph'$; $\eta_{\graph'}\in\Q$ be given by equation (\ref{eq:recbiconn}). Recall that  by equation (\ref{eq:Qgen}) the maps $\ri^{\Db^{\vertex',\lp',0}}$ read as 
\begin{equation}\label{eq:Qbiconn}
\ri^{\Db^{\vertex',\lp',0}}\defeq\sum_{\graph'\in\Vb^{\vertex',\lp',0}}\eta_{\graph'}\,\ri^{\graph'}\,.
\end{equation}   
We proceed to show that a graph isomorphic to $\graph$  is generated by applying  the maps  $\ri^{\Db^{\vertex',\lp',0}}$ to graphs with $\comp-1$ biconnected components occurring in $\Dc^{\vertex-\vertex'+1,\lp-\lp',\ext}=\sum_{\graph^*\in\Vc^{\vertex-\vertex'+1,\lp-\lp',\ext}}\beta_{\graph^*}\graph^*$; $\beta_{\graph^*}\in\Q$ with non-zero coefficient.
Let $\hat{\graph}$ be an arbitrary biconnected component of the graph $\graph$. Let $\hat{V} =\{v_{i_1},\ldots,v_{i_{\uu}}\}\subset V$ be its vertex set, where $i_1<\ldots<i_{\uu}$. Also, let $\lp'$ denote its cyclomatic number.
Contracting the graph $\hat{\graph}$  to the vertex $v_{i_1}$ yields a graph $\contraction(\graph)\in \Vc^{\vertex-\vertex'+1,\lp-\lp',\ext}$ with $\comp-1$ biconnected components. 
Let $\B$ denote the equivalence class containing the graph $\contraction(\graph)$. 
Let $\graph'\in\Vb^{\vertex',\lp',0}$ 
be a biconnected graph isomorphic to $\extc(\hat{\graph})$, where $\extc(\hat{\graph})$ is the graph obtained from $\hat{\graph}$ by erasing all the external edges. 
By induction assumption, there exists a graph in $\B$, say, $\graphH^*$,
so that $\graphH^*\cong\contraction(\graph)$ and $\beta_{\graphH^*}>0$. Let $v_j$ with $j\in\{1,\ldots,\vertex-\vertex'+1\}$ be the vertex of the graph $\graphH^*$ which is mapped to $v_{i_1}$ of $\contraction(\graph)$ by an isomorphism.  
Applying the map $\rj^{\graph'}$ to  
the graph ${\graphH^*}$ yields a linear combination of graphs, one of which, say, $\graphH$,
is isomorphic to $\graph$. That is, $\alpha_{\graphH}>0$ and $\graphH\cong\graph$.
\end{proof}
\begin{lem}
\label{lem:distconn}Fix integers  $\lp\ge0$,   $\vertex>1$ and $\ext\ge\vertex$. 
Let  $\C\subseteq \Vc^{\vertex,\lp,\ext}$ denote an equivalence class.  Let $\graph=(V,K,E)$; $E=\Int\cup\Ext$, 
together with the maps $\mapi$ and $\mape$ denote a graph in $\C$.  Assume that $V\cap\mape(\Ext)=V$.  Let $\Dc^{\vertex,\lp,\ext}=\sum_{\graph\in \Vc^{\vertex,\lp,\ext}}\alpha_\graph\, \graph \in \Q\Vc^{\vertex,\lp,\ext}$ be defined by formula (\ref{eq:recconn}). Then, $\sum_{\graph\in\C}\alpha_{\graph}=1/\vert\aut(\C)\vert$. 
\end{lem}
\begin{proof}
The proof proceeds by induction on the number of biconnected components $\comp$. 
By Theorem \ref{thm:biconn}, the statement holds for all graphs in $\Dc^{\vertex,\lp,\ext}$ with only one biconnected component.
We assume the statement to  hold for graphs in $\Dc^{\vertex,\lp,\ext}$ with 
$\comp-1\ge1$ biconnected components. 
Let the graph $\graph\in\C\subseteq\Vc^{\vertex,\lp,\ext}$ have $\comp$ biconnected components. Let $\edge=\lp+\vertex-1$ denote its internal edge number.
By Lemma~\ref{lem:allconn}, there exists a graph, say, $\graphH\in\C$ which occurs in $\Dc^{\vertex,\lp,\ext}$ with non-zero coefficient. That is, $\graphH\cong\graph$ and $\alpha_{\graphH}>0$.
 Moreover, the graph $\graphH\in\C$ is so that every one of its vertices has at least one (labeled) external edge. Hence,  $\vert\autv(\C)\vert=1$ so that
$\vert\aut(\C)\vert=\vert\aute(\C)\vert$. 
We proceed to show that
$\sum_{\graph\in\C}\alpha_{\graph}=1/\vert\aut(\C)\vert$. To this end, we check from which graphs
with $\comp-1$ biconnected components, the graphs in  the equivalence class   $\C$ are generated  by the recursion formula (\ref{eq:recconn}), and how many times they are generated.  

Choose any one of the $\comp$ biconnected components of the graph $\graphH\in\C$.
Let this be a graph, say, $\hat{\graph}$, with vertex set $\hat{V} =\{v_{i_1},\ldots,v_{i_{\uu}}\}\subseteq V$, where $i_1<\ldots<i_{\uu}$. Also, let  $\lp'$, $\edge'=\lp'+\uu-1$ and $\ext'\ge\uu$ denote its cyclomatic number, internal edge number and external edge number, respectively.
Moreover, let $\A$ denote the equivalence class containing  the graph $\hat{\graph}$. Since this is a subgraph of the graph $\graphH$, $\vert\aut(\A)\vert=\vert\aute(\A)\vert$. 
Contracting the graph $\hat{\graph}$  to the vertex $v_{i_1}$ yields a graph $\contraction(\graphH)\in \Vc^{\vertex-\vertex'+1,\lp-\lp',\ext}$ with $\comp-1$ biconnected components.  Let  $\B$ denote the equivalence class containing $\contraction(\graphH)$.
The graphs in $\B$ have no non-trivial vertex symmetries. Hence, the order of their automorphism group 
 is related to that of the graph $\graphH\in\C$ via
$$\vert\aut(\B)\vert=\frac{\vert\aut(\C)\vert}{\vert\aut(\A)\vert}\,.$$ 
Let $\Dc^{\vertex-\vertex'+1,\lp-\lp',\ext}=\sum_{\graph^*\in\Vc^{\vertex-\vertex'+1,\lp-\lp',\ext}}\beta_{\graph^*}\graph^*\,;\beta_{\graph^*}\in\Q$. 
By induction assumption,  $$\sum_{\graph^*\in\B}\beta_{\graph^*}=\frac{1}{\vert\aut(\B)\vert}\,.$$ 
Now, let $\Db^{\vertex',\lp',0}=\sum_{\graph'\in\Vb^{\vertex',\lp',0}}\eta_{\graph'}\graph'$; $\eta_{\graph'}\in\Q$. 
Let $\graph'\in\Vb^{\vertex',\lp',0}$ 
be a biconnected graph isomorphic to $\extc(\hat{\graph})$. Let $\D\subseteq\Vb^{\vertex',\lp',0}$ denote the equivalence class containing $\graph'$. The order of the automorphism group of the graph $\graph'$ is related to that of $\hat{\graph}$ via $$\vert\aut(\D)\vert=\vert\aut(\A)\vert \cdot \vert\autv(\D)\vert$$ for $\vert\aut(\A)\vert=\vert\aute(\A)\vert=\vert\aute(\D)\vert$.
By Lemma \ref{lem:allconn}, there exists a graph, say, $\graphH^*\in\B$ so that $\graphH^*\cong\contraction(\graphH)$ and $\beta_{\graphH^*}>0$.
Let $v_j$ with $j\in\{1,\ldots,\vertex-\vertex'+1\}$ be the vertex of the graph $\graphH^*$ which is mapped to $v_{i_1}$ of $\contraction(\graphH)$ by an isomorphism.  
Apply the map $\rj^{\graph'}$  to the graph $\graphH^*$. 
Notice that there are $\vert\autv(\D)\vert$ ways to distribute the $\ext'$ external edges assigned to the vertex $v_{j}$ of the graph $\graphH^*$ between all the vertices of the graph  $\shiftj(\graph')$ so as to obtain a graph in the equivalence class $\C\owns \graphH$.  Therefore, there are $\vert\autv(\D)\vert$  graphs in the linear combination 
$\rj^{\graph'}(\graphH^*)$ which are isomorphic to the graph $\graph$. 
Clearly, the map  $\rj^{\graph'}$  produces a graph isomorphic to $\graphH$ from the graph  $\graphH^*$  with coefficient $\alpha^*_{\graphH}=
\beta_{\graphH^*}\in\Q$. 
Now, 
formula (\ref{eq:recconn}) prescribes to  
apply the maps $\ri^{\graph'}$ to the vertex which is mapped to $v_{i_1}$ by an isomorphism  of every graph 
in the equivalence class $\B$  occurring 
in $\Dc^{\vertex-\vertex'+1,\lp-\lp',\ext}$ with non-zero coefficient. 
Therefore,
\begin{eqnarray*}
\sum_{\graph\in\C}\alpha^*_{\graph} & = &\vert\autv(\D)\vert\cdot\sum_{\graph^*\in\B}\beta_{\graph^*}\\
& = & \frac{\vert\autv(\D)\vert}{\vert\aut(\B)\vert}\\
& = & \frac{\vert\autv(\D)\vert\cdot \vert\aut(\A)\vert}{\vert\aut(\C)\vert}\\
& = & \frac{\vert\aut(\D)\vert}{\vert\aut(\C)\vert}\,,
\end{eqnarray*}
where the factor $\vert\autv(\D)\vert$ on the right hand side of the first equality, is due to the fact that every graph (with non-zero coefficient) in the equivalence class $\B$ generates $\vert\autv(\D)\vert$ graphs in $\C$.
Hence, according to Theorem \ref{thm:biconn} and formulas (\ref{eq:recconn}) and (\ref{eq:Qbiconn}), the contribution to $\sum_{\graph\in\C}\alpha_{\graph}$
 is $\edge'/ (\edge\cdot \vert\aut(\C)\vert)$. Distributing
this factor between the $\edge'$ internal edges of the graph $\hat{\graph}$ yields
$1/(\edge\cdot \vert\aut(\C)\vert)$ for each edge. Repeating the same consideration for every biconnected component of the graph $\graph$ yields that every edge of each biconnected component adds $1/(\edge\cdot \vert\aut(\C)\vert)$ to $\sum_{\graph\in\C}\alpha_{\graph}$.

\smallskip

We conclude that every one of the $\edge$ internal edges of the graph $\graph$  
contributes $1/(\edge\cdot \vert\aut(\C)\vert)$ to  $\sum_{\graph\in\C}\alpha_{\graph}$. Hence, the overall contribution  is exactly $1/\vert\aut(\C)\vert$. This
completes the proof.
\end{proof}
$\Dc^{\vertex,\lp,\ext}$ satisfies the following property.
\begin{lem}\label{lem:facconn}
Fix integers $\ext\ge0$, $\lp\ge0$ and  $\vertex>1$. 
Let $\Dc^{\vertex,\lp,\ext}=\sum_{\graph\in \Vc^{\vertex,\lp,\ext}}\alpha_\graph\, \graph \in \Q\Vc^{\vertex,\lp,\ext}$ be defined by formula (\ref{eq:recconn}). 
Moreover, let $K'\subset\mathcal{K}$ be a finite set so that $K\cap K'=\emptyset$.  Let $\Ext'\subseteq [K']^\two$; \emph{$\ext'\defeq\card(\Ext')$}. Assume that  the elements of $\Ext'$ satisfy  $\{\legs_a,\legs_{a'}\}\cap\{\legs_{b},\legs_{b'}\}=\emptyset$.  Also, let  $L'=\{x_{\ext+1},\ldots,x_{\ext+\ext'}\}$ be a label set so that  $L\cap L'=\emptyset$. 
Let  $l':\Ext'\rightarrow [K',L']$ be a labeling of the elements of $\Ext'$. 
 Then,  $\Dc^{\vertex,\lp,\ext+\ext'}=\extb_{\Ext',V}(\Dc^{\vertex,\lp,\ext})$. 
\end{lem}
\begin{proof}
Let $V^*=\{v_i, v_{\vertex-\vertex'+2},\ldots,v_{\vertex}\}\subseteq V$ be the vertex set of all graphs in $\shift(\Db^{\vertex',\lp',0})$; $i\in\{1,\ldots,\vertex-\vertex'+1\}$.
Clearly, $\extb_{\Ext',V}=\sum_{\edgese\subseteq\Ext'}\extb_{\Ext'\backslash\edgese,V\backslash V^*}\circ\extb_{\edgese,V^*}:\Q \Vc^{\vertex,\lp,\ext}\to\Q \Vc^{\vertex,\lp,\ext+\ext'}$. 
Furthermore, $\extb_{\edgese,V^*}\circ\ri^{\Db^{\vertex',\lp',0}}=\ri^{\Db^{\vertex',\lp',0}}\circ\extb_{\edgese,\{v_i\}}:\Q \Vc^{\vertex-\vertex'+1,\lp-\lp',\ext}\rightarrow\Q \Vc^{\vertex,\lp,\ext+\ext^*}$, where $\ext^*=\card(\edgese)$. 
Therefore, the equality $\Dc^{\vertex,\lp,\ext+\ext'}=\extb_{\Ext',V}(\Dc^{\vertex,\lp,\ext})$ follows immediately from the recursive definition (\ref{eq:recconn}).
 \end{proof}

\begin{lem}
\label{lem:anyclass} Fix integers $\ext\ge0$, $\lp\ge0$ and  $\vertex>1$. 
 Let  $\C\subseteq \Vc^{\vertex,\lp,\ext}$ denote an arbitrary equivalence class.
   Let $\Dc^{\vertex,\lp,\ext}=\sum_{\graph\in \Vc^{\vertex,\lp,\ext}}\alpha_\graph\, \graph \in \Q\Vc^{\vertex,\lp,\ext}$ be defined by formula (\ref{eq:recconn}). Then, $\sum_{\graph\in\C}\alpha_{\graph}=1/\vert\aut(\C)\vert$. 
\end{lem}
\begin{proof}
The proof is the same as that of Lemma 10 of \cite{Me:classes} (see also  Lemma 10 and Theorem 10 of \cite{MeOe:npoint} and \cite{MeOe:loop}, respectively). 
\end{proof} 
This completes the proof of Theorem \ref{thm:conn}.
\end{proof}

\subsubsection{Examples}\label{sec:appconn}
The present section overlaps Section 5.3.3 of \cite{Me:classes}. We show the result of computing all mutually non-isomorphic connected graphs without external edges as  contributions to $\Dc^{\vertex,\lp,0}$ via formula  (\ref{eq:recconn}) up to order $2\le\vertex+\lp\le5$.
The coefficients in front  of graphs are the
inverses of the orders of their
groups of automorphisms.
\vspace{1cm}\\
\includegraphics{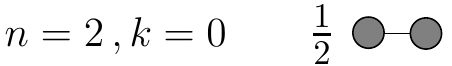}
\vspace{1cm}\\
\includegraphics{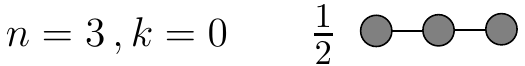}
\vspace{1cm}\\
\includegraphics{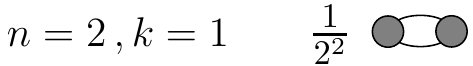}
\vspace{1cm}\\
\includegraphics{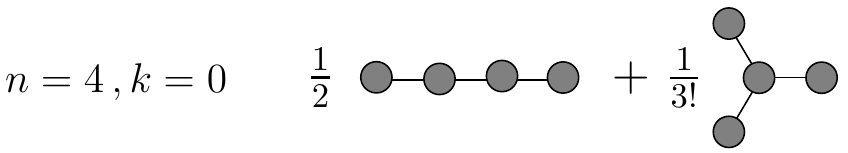}
\vspace{1cm}\\
\includegraphics{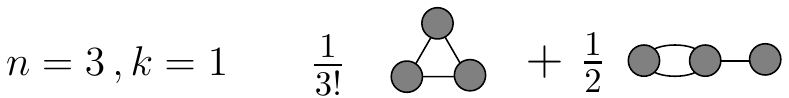}
\vspace{1cm}\\
\includegraphics{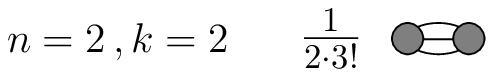}
\vspace{1cm}\\
\includegraphics{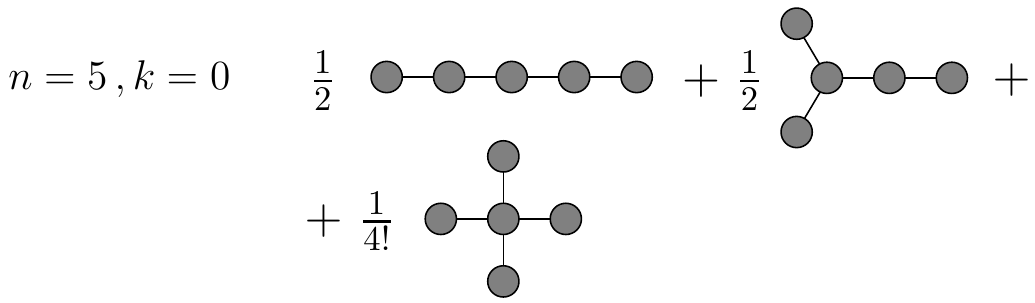}
\vspace{1cm}\\
\includegraphics{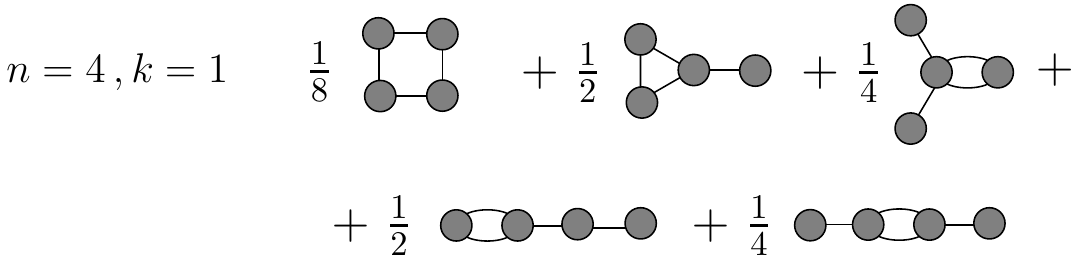}
\vspace{1cm}\\
\includegraphics{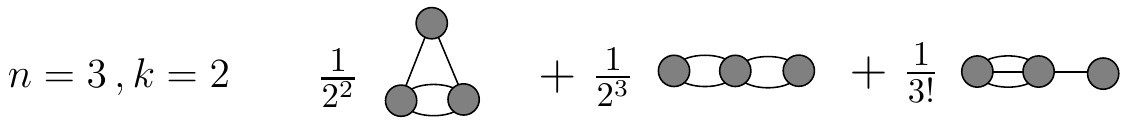}
\vspace{1cm}\\
\includegraphics{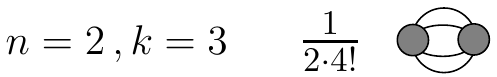}

\subsection{2-edge connected graphs}
\label{sec:edgeconn}
By Lemma \ref{lem:edgeconntoedgeconn}, Theorem \ref{thm:conn} generalizes straightforwardly to $\two$-edge connected graphs. 
\begin{thm}\label{thm:edgeconn}
Fix an integer $\ext\ge0$.   
For all integers $\lp>0$ and $\vertex> 1$,
define $\De^{\vertex,\lp,\ext} \in \Q\Ve^{\vertex,\lp,\ext}$
by the following recursion relation:
\begin{itemize}
\item
$\De^{\two,\lp,\ext}\defeq\Db^{\two,\lp,\ext}$; 
\item
\begin{eqnarray}\nonumber
\lefteqn{\De^{\vertex,\lp,\ext}\defeq
\Db^{\vertex,\lp,\ext}+\frac{1}{\lp+\vertex-1}\cdot}\\\label{eq:recedgeconn}
&&\sum_{\lp'=1}^{\lp-1}\sum_{\vertex'=\two}^{\vertex-1}\sum_{i=1}^{\vertex-\vertex'+1}\biggl((\lp'+\vertex'-1)\ri^{\Db^{\vertex',\lp',0}}(\De^{\vertex-\vertex'+1,\lp-\lp',\ext})\biggr), \vertex>\two\,.
 \end{eqnarray}
\end{itemize}
Then, for fixed values of  $\vertex$ and $\lp$, $\De^{\vertex,\lp,\ext}=\sum_{\graph\in \Ve^{\vertex,\lp,\ext}}\alpha_\graph\, \graph$;  $\alpha_\graph\in\Q$  for all $\graph\in\Ve^{\vertex,\lp,\ext}$. Moreover, given an arbitrary equivalence class $\C\subseteq\Ve^{\vertex,\lp,\ext}$, the following holds: (i) There exists $\graph\in\C$ so that $\alpha_\graph>0$; (ii) $\sum_{\graph\in\C}\alpha_\graph=1/\vert\aut(\C)\vert$. 
\end{thm}

\subsubsection{Examples}\label{sec:appedgeconn}
We show the result of computing all mutually non-isomorphic $\two$-edge connected graphs without external edges as  contributions to $\De^{\vertex,\lp,0}$ via formula  (\ref{eq:recedgeconn}) up to order $3\leq\vertex+\lp\le6$. The coefficients in front  of graphs are the
inverses of the orders of their groups of automorphisms. 
\vspace{1cm}\\
\includegraphics{figures/noloopsn2k1}
\vspace{1cm}\\
\includegraphics{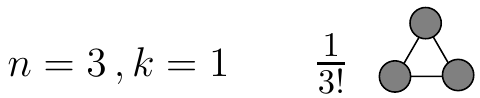}
\vspace{1cm}\\
\includegraphics{figures/noloopsn2k2}
\vspace{1cm}\\
\includegraphics{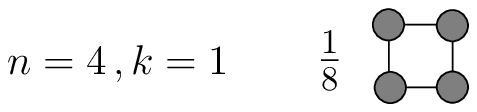}
\vspace{1cm}\\
\includegraphics{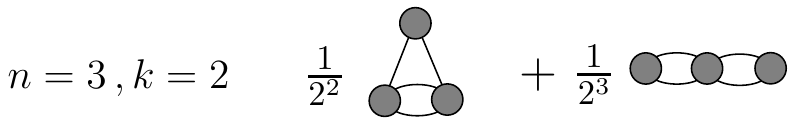}
\vspace{1cm}\\
\includegraphics{figures/noloopsn2k3}
\vspace{1cm}\\
\includegraphics{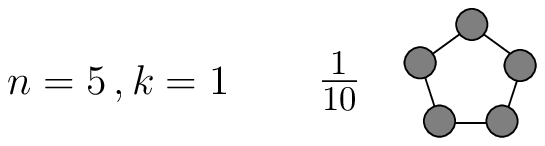}
\vspace{1cm}\\
\includegraphics{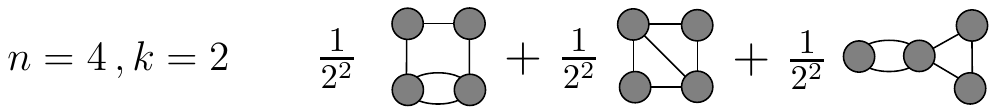}
\vspace{1cm}\\
\includegraphics{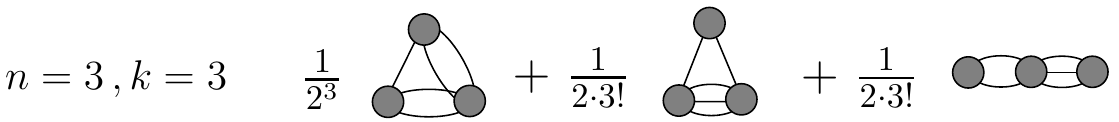}
\vspace{1cm}\\
\includegraphics{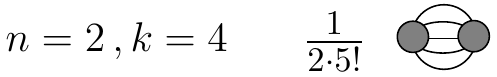}

\subsubsection{Algorithmic considerations}
\label{sec:algo}
We briefly discuss some of the algorithmic implications of the recursive definition (\ref{eq:recedgeconn}). In the present section,  only graphs without external edges are considered for these may be added via the maps $\extb_{\Ext,V}$. 

\medskip

An important algorithmic aspect is to determine \emph{a priori} the nature of the biconnected components of the $\two$-edge connected graphs generated by formula (\ref{eq:recedgeconn}).
 In this context, the most straightforward simplification is to restrict the formula to graphs whose biconnected components are cycles. Let $C_{\vertex}$ denote a cycle with $\vertex$ vertices. Clearly, from formula (\ref{eq:recbiconn}) $\Db^{\vertex,1,0}=1/(\two\vertex)C_{\vertex}$, where for simplicity all graphs in the same equivalence class are identified as the same. Therefore, formula (\ref{eq:recedgeconn}) specializes to $\two$-edge connected graphs with the aforesaid property as follows: 
\begin{eqnarray*}
\Dec^{\two,1,0}&\defeq&\frac{1}{\two^\two}  C_\two\,;\\
\Dec^{\vertex,\lp,0}&\defeq&
\frac{1}{\two\vertex}C_\vertex+\frac{1}{\two(\lp+\vertex-1)}
\sum_{\vertex'=\two}^{\vertex-1}\sum_{i=1}^{\vertex-\vertex'+1}\ri^{C_{\vertex'}}(\Dec^{\vertex-\vertex'+1,\lp-1,0}), \vertex>\two\,,
 \end{eqnarray*}
where $\Dec^{\vertex,\lp,0}\in\Q\Ve^{\vertex,\lp}$ denotes the linear combination of all $\two$-edge connected graphs with $\vertex$ vertices and cyclomatic number $\lp$ so that every biconnected component is a cycle. In addition, suppose that one is only interested in calculating all $\two$-edge connected graphs  whose biconnected components have a minimum vertex  or cyclomatic number, say, $\vmin$ and $\lmin$, respectively. This is clearly obtained by changing the lower and upper limits of the sums over $\vertex'$ and $\lp'$ in formula (\ref{eq:recedgeconn}) to $\vmin$ and $\vertex-\vmin+1$ or $\lmin$ and $\lp$-$\lmin$, respectively.

\section*{Acknowledgements}
The author would like to thank Brigitte Hiller and Christian Brouder  for 
careful reading of the manuscript.
The research was supported through the fellowship SFRH/BPD/48223/2008 provided by the Portuguese Science and Technology Foundation.

\bibliographystyle{abbrv}
\bibliography{bib1001a}

\end{document}